\documentclass[12pt]{amsart}
\usepackage[english]{babel}
\usepackage{txfonts}
\usepackage{fancyhdr}

\usepackage{amsthm,amsmath,amssymb,amsfonts,latexsym,url}

\usepackage{tikz}
\usepackage[all]{xy}
\usetikzlibrary{decorations.markings}
\usetikzlibrary{backgrounds,shapes}
\usepackage{subfig}
  \usepackage{graphicx}
\usepackage{enumerate}
\makeatletter
\let\@@enum@org\@@enum@
\def\@@enum@[#1]{\@@enum@org[\normalfont #1]}
\makeatother
\newtheorem{thm}{Theorem}[section]
\newtheorem{lem}[thm]{Lemma}
\newtheorem{cor}[thm]{Corollary}
\newtheorem{prop}[thm]{Proposition}
\newtheorem{con}[thm]{Conjecture}

\newtheorem*{claim*}{Claim}

\theoremstyle{remark}

\newtheorem*{hint}{Hint}

\theoremstyle{definition}
\newtheorem{defn}[thm]{Definition}
\newtheorem{prob}{Problem}[section]

\newcommand\be{\begin{enumerate}}
\newcommand\ee{\end{enumerate}}
\newcommand\bp{\begin{proof}}
\newcommand\ep{\end{proof}}
\newcommand\bpp{\begin{prop}}
\newcommand\epp{\end{prop}}
\newcommand\bpb{\begin{prob}}
\newcommand\epb{\end{prob}}
\newcommand\bd{\begin{defn}}
\newcommand\ed{\end{defn}}
\newcommand\bh{\begin{hint}}
\newcommand\eh{\end{hint}}

\newcommand\C{\mathbb{C}}

\newcommand\Ind{\operatorname{Ind}}

\newcommand\Mod{\operatorname{Mod}}

\newcommand\yt{\widetilde}

\def\thetitle{{A representation theoretic characterization of simple closed curves on a surface}}
\def\theshorttitle{}
\usepackage{hyperref}

\begin{document}
\raggedbottom
\title[\theshorttitle]\thetitle
\date{\today}

\author{Thomas Koberda}
\address{Department of Mathematics, University of Virginia, Charlottesville, VA 22904-4137, USA}
\email{thomas.koberda@gmail.com}

\author{Ramanujan Santharoubane}
\address{Department of Mathematics, University of Virginia, Charlottesville, VA 22904-4137, USA}
\email{ramanujan.santharoubane@gmail.com}

\begin{abstract}
We produce a sequence of finite dimensional representations of the fundamental group $\pi_1(S)$ of a closed surface where all simple closed curves act with finite order, but where each non--simple closed curve eventually acts with infinite order. As a consequence, we obtain a representation theoretic algorithm which decides whether or not a given element of $\pi_1(S)$ has a representative in its free homotopy class which is a simple closed curve. The construction of these representations combines ideas from TQFT representations of mapping class groups with effective versions of LERF for surface groups.
\end{abstract}

\maketitle

\section{Introduction}

Let $S=S_{g}$ be a closed and orientable surface of genus $g\geq 2$. A \emph{simple closed curve} on $S$ is a homotopically essential embedding of the circle $S^1$ into $S$. Fixing a basepoint on $S$, an element $1\neq \gamma\in\pi_1(S)$ is called \emph{simple} if it has a simple closed curve in its free homotopy class, and \emph{non--simple} otherwise. A homotopy class of loops $\gamma\colon S^1\to S$ is called a \emph{proper power} if there is a loop $\gamma_0\colon S^1\to S$ and an $n>1$ such that $\gamma=\gamma_0^n$, i.e. $\gamma_0$ concatenated with itself $n$ times. In this article, we propose a representation theoretic characterization of the simple elements in $\pi_1(S)$, which thus relates a topological property of a conjugacy class in $\pi_1(S)$ with the representation theory of $\pi_1(S)$. Precisely:

\begin{thm}\label{thm:main}
Let $X$ be a fixed finite generating set for $\pi_1(S)$, and let $\ell_X$ be the corresponding word metric on $\pi_1(S)$. There exists an explicit sequence of finite dimensional complex representations $\{\rho_i\}_{i\geq 1}$ such that:
\begin{enumerate}
\item
For each simple element $1\neq \gamma\in\pi_1(S)$ and each $i$, the linear map $\rho_i(\gamma)$ has finite order bounded by a constant which depends only on $i$;
\item
An element $1\neq \gamma\in\pi_1(S)$ which is not a proper power is non--simple if and only if $\rho_i(\gamma)$ has infinite order for all $i\geq \ell_X(\gamma)$.
\end{enumerate}
\end{thm}

\begin{cor}\label{cor:representation}
Let $\gamma\in\pi_1(S)$ be non--simple and not a proper power. Then there is a finite dimensional linear representation $\rho$ of $\pi_1(S)$ under which all simple elements have finite order, but where $\rho(\gamma)$ has infinite order.
\end{cor}

Algebraic and algorithmic characterizations of simple closed curves on surfaces have been studied by many authors for quite some time. The first group--theoretic characterization of simplicity of elements in $\pi_1(S)$ was given by Zieschang (see ~\cite{Z1965},~\cite{Z1969}). Algorithms, both algebraic and geometric in nature, were given by Chillingworth in~\cite{Chillingworth1} and~\cite{Chillingworth2}, Turaev--Viro~\cite{TuraevViro}, Birman--Series~\cite{BirmanSeries} with extensions by Cohen--Lustig~\cite{CohenLustig}, Hass--Scott~\cite{HassScott}, Arettines~\cite{Arettines}, Cahn~\cite{Cahn}, and Chas--Krongold in~\cite{ChasKrongold1} and~\cite{ChasKrongold2}. A homological characterization using finite covering spaces was given by Boggi~\cite{Boggi}. The foundational work of Goldman~\cite{Goldman} has guided much of the progress in the study of algebraic characterizations of simple closed curves (see also~\cite{Turaev}).

\begin{cor}\label{cor:algorithm}
The representations $\{\rho_i\}_{i\geq 1}$ in Theorem~\ref{thm:main} are computable. In particular, there exists an algorithm which, by computing a certain representation of $\pi_1(S)$, decides whether or not an element $1\neq \gamma\in\pi_1(S)$ is simple.
\end{cor}

The representations in Theorem \ref{thm:main} and Corollary \ref{cor:representation} come from TQFT representations of mapping class groups, combined with certain canonical representation theoretic constructions. To the authors' knowledge, Theorem~\ref{thm:main} gives the first representation theoretic and the first topological quantum field theoretic characterization of simple closed curves on a surface.

We remark that Theorem \ref{thm:main} and its corollaries all hold for compact orientable surfaces with boundary. We have limited our discussion to closed surfaces in order to make some of the statements cleaner, but the arguments work for surfaces with boundary in a manner which is essentially unchanged.

\section{Acknowledgements}
The authors thank J. Aramayona, V. Krushkal, G. Masbaum, and H. Parlier for helpful discussions. The authors are indebted to an anonymous referee for helpful comments which improved the paper. The first author is partially supported by Simons Foundation Collaboration Grant number 429836.

\section{Proof of the main result}

In this section, we gather the main facts needed to establish Theorem~\ref{thm:main}, relegating discussion to later sections. Throughout this section and the remainder of the paper, we will equip $S$ with an (unless otherwise stated, arbitrary) hyperbolic metric.

\subsection{TQFT representations of surface groups and figure eight loops}

Let $\gamma\subset S$ be closed geodesic with exactly one self--intersection, so that $\gamma$ can be termed a \emph{figure eight loop}. Fixing a basepoint for $S$ on $\gamma$ (say the self--intersection point), we can orient $\gamma$ and identify it with an element of $\pi_1(S)$. Note that the requirement that $\gamma$ be geodesic automatically guarantees that this figure eight loop is \emph{essential}, i.e. that the two simple subloops of $\gamma$ generate a copy of the free group $F_2<\pi_1(S)$. The following is a straightforward consequence of the main technical result of the authors~\cite{KoberdaSantharoubane}:

\begin{lem} \label{lem:KobSanth}
Let $S_g=S$ a closed surface of genus $g\geq 2$. There is a linear representation $\rho\colon\pi_1(S)\to\text{GL}_n(\C)$ such that:
\begin{enumerate}
\item
If $\gamma\in\pi_1(S)$ has an embedded essential figure eight loop as its geodesic representative, then $\rho(\gamma)$ has infinite order;
\item
The image of each simple element in $\pi_1(S)$ under $\rho$ has finite order bounded by $k$, where $k$ depends only on $\rho$.
\end{enumerate}
\end{lem}

In the interest of brevity, we will not describe the representation $\rho$ in any detail. However, we will note that the representation $\rho$ is computable and can be written down explicitly. See~\cite{KoberdaSantharoubane}.

\begin{proof}[Proof of Lemma \ref{lem:KobSanth}]
 In \cite{KoberdaSantharoubane}, we considered a sequence of representations $$ \rho_{p} :  \pi_1(S_g) \to \mathrm{PGL}_{d(p,g)}(\mathbb{C})$$ indexed by odd integers $p \geq 3$. 
 
It is straightforward to adjust the proof of Theorem 1.1 of~\cite{KoberdaSantharoubane} to establish the following: if $\gamma$ is an arbitrary figure eight loop then there exists an integer $p_0(\gamma)$ such that for $p \geq p_0(\gamma)$, we have that $\rho_p(\gamma)$ has infinite order.

These representations were obtained from restrictions of quantum $\mathrm{SO}(3)$ representations of mapping class groups. This clearly implies that if $\rho_p(\gamma)$ has infinite order for some $p$, then $\rho_p(\gamma')$ has infinite order for all $\gamma'$ in the mapping class group orbit of $\gamma$. 

Now since there are only finitely many figure eight loops in $\pi_1(S_g)$ up to the action of the mapping class group, we can choose one integer $p_0$ such that if $\gamma$ is an arbitrary figure eight loop then $\rho_{p_0}(\gamma)$ is an infinite order element.

Finally, the image of any simple element in $\pi_1(S_g)$ under $\rho_{p_0}$ has order at most $2p_0$. 
\end{proof}

\subsection{Inducing up}

We will require some well--known fact from representation theory, which we gather here for the convenience of the reader. We will restrict to complex representations, though the discussion works over any field of characteristic zero.

Let $G$ be a group, let $H$ be a finite index subgroup of $G$, and let $V$ be a finite dimensional complex representation of $H$. The \emph{induced representation} of $G$ is a canonical way to turn $V$ into a finite dimensional representation of $G$. Precisely, we take \[\Ind_H^G V:=\C[G]\otimes_{\C[H]}V.\] It is standard that the complex dimension of $\Ind_H^G V$ is given by $[G:H]\cdot\dim V$. Frobenius reciprocity, suitably generalized to infinite groups, guarantees that if $h\in H$ acts with infinite order on $V$ then $h$ also acts with infinite order on $\Ind_H^G V$.

\begin{lem}\label{lem:induction}
Let $H<\pi_1(S)$ be a finite index subgroup, classifying a finite cover $S'\to S$. Let $V$ be a finite dimensional representation of $H$ such that each simple element of $H=\pi_1(S')$ acts with finite order. Then if $1\neq \gamma\in\pi_1(S)$ is simple, we have that $\gamma$ acts with finite order on $\Ind_H^G V$.
\end{lem}

\begin{proof}
The lemma follows from a straightforward computation, using the fact that if $\gamma\in\pi_1(S)$ is simple and $S'\to S$ is a finite cover of $S$, then some power of $\gamma$ lifts to a simple element of $\pi_1(S')$. We omit the details in the interest of brevity.
\end{proof}

\subsection{Subgroup separability}

Surface groups enjoy \emph{subgroup separability} (also called \emph{LERF}; see Section~\ref{sec:fig8}), which implies the following fact:

\begin{lem}\label{lem:P LERF}
Let $\gamma\subset S$ be a closed geodesic with at least one self--intersection. Then there is a finite cover $S'\to S$ such that $\gamma$ lifts to a figure eight loop in $S'$.
\end{lem}

We will discuss subgroup separability and give a proof of Lemma~\ref{lem:P LERF} in Section~\ref{sec:fig8}.

\subsection{Proof of the main result}

Theorem~\ref{thm:main} follows easily from the facts in this section:

\begin{proof}[Proof of Theorem~\ref{thm:main}]
Let $1\neq\gamma\in\pi_1(S)$ be a non--simple element which is not a proper power, and represent $\gamma$ by a geodesic on $S$. By Lemma~\ref{lem:P LERF}, there is a finite cover $S'$ of $S$ to which $\gamma$ lifts and where it becomes a figure eight loop. By Lemma~\ref{lem:KobSanth}, there is a linear representation $\rho$ of $\pi_1(S')$ such that $\rho(\gamma)$ has infinite order, and such that each simple element of $\pi_1(S')$ has (uniformly bounded) finite order. Inducing $\rho$ to $\pi_1(S)$, we have that $\gamma$ still acts with infinite order. By Lemma~\ref{lem:induction}, all simple elements of $\pi_1(S)$ will still act with finite order.

Now, there are only finitely many elements in $\pi_1(S)$ of a given word length in the word metric $\ell_X$ coming from a finite generating set $X$. For each non--simple $\gamma\in\pi_1(S)$, write $\rho_{\gamma}$ for the representation produced in the previous paragraph, and write \[\rho_i=\bigoplus_{\gamma\in\pi_1(S),\, \ell_X(\gamma)\leq i} \rho_{\gamma},\] where $\gamma$ ranges over non--simple elements of $\pi_1(S)$.

Thus, if $\gamma$ is non--simple and not a proper power with $\ell_X(\gamma)\leq i$, then $\rho_i(\gamma)$ has infinite order. Since every factor in the direct sum decomposition of $\rho_i$ sends simple elements of $\pi_1(S)$ to uniformly bounded finite order elements, we obtain the claimed result.
\end{proof}

\section{Immersed figure eight loops}\label{sec:fig8}
Let $G$ be a group. We say that $G$ is \emph{locally extended residually finite} or \emph{LERF} if every finitely generated subgroup of $G$ is closed in the profinite topology on $G$. The topological meaning of LERF is given by the following well--known fact:

\begin{lem}\label{lem:lerf}
Let $X$ be a simplicial complex such that $\pi_1(X)$ is LERF, let $K$ be a finite simplicial complex, and let $f\colon K\to X$ be a simplicial map such that some lift \[\tilde{f}\colon\widetilde{K}\to\widetilde{X}\] of $f$ to the universal covers of $K$ and $X$ respectively is an embedding. Then there exists a finite cover $X'\to X$ and a lift $f'\colon K\to X'$ of $f$ such that $f'$ is an embedding.
\end{lem}

Lemma~\ref{lem:lerf} is standard and we do not provide a proof here. A famous result of P. Scott~\cite{Scott} says that closed surface groups are LERF.

Let $E\cong S^1\vee S^1$ be a figure eight and let $S$ be a surface as before. A map $f\colon E\to S$, is called \emph{$\pi_1$--injective} if $f$ induces an injective map on fundamental groups and if $f$ is a local homeomorphism.

\begin{lem}\label{lem:figure8}
Let $\gamma\subset S$ be a closed geodesic with at least one self--intersection. Then $\gamma$ is the image of a $\pi_1$--injective map $f\colon E\to S$.
\end{lem}
\begin{proof}
Choose an arbitrary orientation on $\gamma$ and an arbitrary self--intersection point $p\in\gamma$. In a small neighborhood $U$ of $p$, we have that $\gamma\cap U$ is a pair of transversely intersecting oriented arcs $\alpha$ and $\beta$. We start at $p$ and travel along $\alpha$ in the direction of the orientation. We follow $\gamma$ until we return to $p$ for the first time. Note that this will happen along the arc $\beta$, since geodesics are uniquely determined by a point and a direction. We denote the resulting loop on $S$ by $\gamma_1$. Continuing from $p$ along $\beta$, we eventually return to $p$ for a second time along $\alpha$, tracing another closed loop $\gamma_2$. We have thus decomposed $\gamma$ as a union of two distinct closed curves on $S$ which meet at $p$. Since $\gamma$ is a geodesic and since $\gamma_1$ and $\gamma_2$ are distinct loops based at $p$, we have that the homotopy classes of $\gamma_1$ and $\gamma_2$ generate a free subgroup of $\pi_1(S,p)$. Therefore, $\gamma=\gamma_1\cup\gamma_2$ is the image of a $\pi_1$--injective map $E\to S$ sending the wedge point of $E$ to $p$.
\end{proof}

We have the following lemma which implies Lemma~\ref{lem:P LERF}:

\begin{lem}\label{lem:P LERF2}
Let $\gamma\subset S$ be a closed geodesic with at least one self--intersection, and let $f\colon E\to S$ be a $\pi_1$--injective map with image $\gamma$. Then there is a finite cover $S'\to S$ and a lift $f'\colon E\to S'$ of $f$ which is an embedding. In particular, $\gamma$ lifts to a figure eight loop in $S'$.
\end{lem}
\begin{proof}
Let $p$ be a self--intersection of $\gamma$, which we use as a basepoint for $\pi_1(S)$. We have that $\pi_1(E)\cong F_2$, and we observed that $f_*(\pi_1(E))$ is a copy of $F_2<\pi_1(S)$. Moreover, it is easy to check that since $\gamma$ is a geodesic, the map $\yt{f}\colon \yt{E}\to\yt{S}$ induced on universal covers is an embedding. The claim of the lemma follows from Scott's Theorem that surface groups are LERF.
\end{proof}

\section{An effective version of the main result}
In this section, we give effective estimates of results relevant to Theorem~\ref{thm:main}, thus establishing Corollary \ref{cor:algorithm}.

\subsection{Proper powers}

It is advantageous for us to assume that a given element $\gamma\in\pi_1(S)$ is not a proper power. In order to legitimize this assumption, we have the following lemma which is easy to prove using the following three well--known facts:

\begin{enumerate}
\item
If $1\neq\gamma\in\pi_1(S)$ then the centralizer of $\gamma$ is cyclic.
\item
The word growth with respect to any finite generating set is exponential.
\item
The Dehn function of $\pi_1(S)$ is linear.
\end{enumerate}

\begin{lem}\label{lem:proper power}
There is an algorithm which decides if $\gamma\in\pi_1(S)$ is a proper power, which has at most exponential complexity in $\ell_X(\gamma)$.
\end{lem}

\subsection{Hyperbolic length versus word length}

The following lemma is an easy consequence of the standard fact that $\pi_1(S)$ is quasi-isometric to hyperbolic space, and we omit its proof:

\begin{lem}\label{lem:qi}
Let $\gamma\in\pi_1(S)$, and assume that $\gamma$ is not a proper power. Let $\ell_{\gamma}$ denote the length of a geodesic representative for $\gamma$, and let $\ell_X(\gamma)$ denote the length of $\gamma$ in the generating set $X$. Then for some constant $\lambda=\lambda(X,S)>0$, we have $\ell_{\gamma}\leq \lambda\cdot\ell_X(\gamma)$.
\end{lem}

\subsection{Length and intersection}

It is a standard fact from surface topology that curves of a given length have self--intersection number bounded above by a quadratic function in the length. Precisely, we have the following:

\begin{lem}\label{lem:intersect}
Let $X$ be a finite generating set for $\pi_1(S)$ and let $\gamma\in\pi_1(S)$. Then there is a constant $C=C(X,S)$ such that the geodesic representative for $\gamma$ has at most $C\cdot (\ell_X(\gamma))^2$ self--intersections.
\end{lem}

Lemma \ref{lem:intersect} is proved in a slightly different form by Malestein--Putman (See~\cite{MalesteinPutman}, Lemma 3.1). More precise self--intersection bounds are also studied in~\cite{CohenLustig},~\cite{ChasPhillips1}, and~\cite{ChasPhillips2}. Since Lemma~\ref{lem:intersect} follows from standard combinatorial methods, we will omit a proof.

\subsection{Effective LERF}

P. Patel has established effective versions of Scott's Theorem that closed surface groups are LERF.
To state the effective version, let $S$ be a compact surface of negative Euler characteristic. Following Patel, we give $S$ the \emph{standard} hyperbolic metric, one which comes from a tiling by right-angled hyperbolic pentagons. If $H<\pi_1(S)$ is an infinite index, finitely generated subgroup, we let $X\to S$ be the corresponding cover and $C(X)$ the \emph{convex core} of $X$. That is to say, $C(X)$ is the smallest, closed, convex subsurface of $X$ with geodesic boundary for which the inclusion into $X$ is a homotopy equivalence. Note that $H$ is a finitely generated free group of rank $n$. Let $\beta$ be the total length of the geodesic boundary of $C(X)$, which is finite since $H$ has finite rank.

\begin{thm}[\cite{Patel}, Theorem 7.1]\label{thm:patel}
Suppose $n\geq 2$, let $\gamma\in\pi_1(S)\setminus H$, and let $\ell_{\gamma}$ be the length of the geodesic representative of $\gamma$. There exists a finite index subgroup $K<\pi_1(S)$ containing $H$ but not $\gamma$ such that \[[\pi_1(S):K]<4n-4+\frac{2\sinh[d\cdot (\ell_{\gamma}/e+2)]}{\pi}\cdot \beta,\] where $d$ and $e$ are fixed positive constants.
\end{thm}

\subsection{Making Theorem~\ref{thm:main} effective}

Let $1\neq\gamma\in\pi_1(S)$ be given and let $X$ be a fixed finite generating set for $\pi_1(S)$. We can effectively check if $\gamma$ is a proper power by Lemma~\ref{lem:proper power}, and replace it by a root if necessary. Lemma~\ref{lem:intersect} combined with Theorem~\ref{thm:patel} together can be used to give effective control over the degree of a cover $S'\to S$ such that $\gamma$ lifts to a figure eight loop in $S'$, provided $\gamma$ is non--simple. Thus if $\gamma$ lifts to a figure eight loop on some finite cover of $S$ then it does so on one of finitely many covers $\{S_1,\ldots,S_N\}$ of $S$, where the degree of each such cover is bounded by a computable function in $\ell_X(\gamma)$.

Lemma~\ref{lem:KobSanth} effectively computes a representation $\rho_i$ of $\pi_1(S_i)$ for each $i$, where all simple elements of $\pi_1(S_i)$ act with finite order but where $\gamma$ acts with infinite order whenever $\gamma$ lifts to a figure eight loop on $S_i$.  Finite induction as in Lemma~\ref{lem:induction} is evidently computable. Finally, there are only finitely many elements of $\pi_1(S)$ of a given length. We may therefore consider the direct sum of the representations $\{\rho_1,\ldots,\rho_N\}$ as above for each $\gamma\in\pi_1(S)$ of fixed length with respect to $X$. We thus produce a representation $\rho$ of $\pi_1(S)$ which detects the non--simple elements of $\pi_1(S)$ of length $\ell_X(\gamma)$ as precisely the ones with infinite order under $\rho$. This establishes Corollary~\ref{cor:algorithm}.

\section{Remarks on the AMU Conjecture for surface groups}

In this final section, we explain the motivation behind the work in this paper. The starting point for us is to use certain representations of surface groups under which figure eight loops have infinite order, and under which simple loops have finite order. In order to prove the results of this paper, we only need the existence of one such representation. However, each representation used here is part of an infinite sequence of representations $$\rho_{p}\colon \pi_1(S) \to \mathrm{PGL}_{d_p}(\C)$$ indexed by odd integers $p$, as considered in~\cite{KoberdaSantharoubane}. For the present work, we only need the following precise statement: if $\gamma$ is geodesic figure eight loop, then $\rho_p(\gamma)$ has infinite order for $p$ big enough (cf. Lemma \ref{lem:KobSanth}).  In general, it is expected that any non--simple element of $\pi_1(S)$ which is not a proper power will have infinite order under $\rho_p$ for all but finitely many values of $p$. More precisely :
\begin{con}[AMU Conjecture for surface groups] \label{AMU2}
If $1\neq \gamma \in \pi_1(S)$ is a non--simple element which is not a proper power then $\rho_p(\gamma)$ has infinite order for $p\gg 0$.
\end{con}
This conjecture is implied by the AMU Conjecture as stated by Andersen, Masbaum, and Ueno~\cite{AMU}. We will briefly explain how this works.

Let $x_0 \in S$ be a fixed marked point and let $\Mod^1(S)$ be the mapping class group of $S$ fixing the marked point $x_0$. The Witten--Reshetikhin--Turaev $\mathrm{SO}(3)$ topological quantum field theory gives a representation 
 $$\tilde{\rho}_{p} : \Mod^1(S) \to \mathrm{PGL}_{d_p}(\C)$$ for each odd integer $p \geq 3$. Using the Birman exact sequence, we can view $\pi_1(S)$ as a subgroup of $\Mod^1(S)$. This allows us to define a representation $\rho_p$ by restricting $\tilde{\rho}_p$ to $\pi_1(S)$. The AMU conjecture as stated in~\cite{AMU} reads as follows:
\begin{con}\label{AMU1}
If $\phi \in \Mod^1(S)$ has a pseudo-Anosov piece then $\tilde{\rho}_p(\phi)$ has infinite order for $p\gg 0$.

\end{con}

To see how Conjecture \ref{AMU1} implies Conjecture \ref{AMU2}, we use a result of Kra (see Theorem 1.1 of~\cite{Kra}): if $1\neq\gamma \in \pi_1(S)$ is non--simple and not a proper power then the corresponding mapping class in $\Mod^1(S)$ is pseudo-Anosov on the subsurface of $S$ filled by $\gamma$. It is now clear that Conjecture \ref{AMU1} implies Conjecture \ref{AMU2}.

We conclude with some progress towards resolving Conjecture~\ref{AMU1} due to J. March\'e and the second author in~\cite{MarchSanth}. Precisely, they show that Conjecture~\ref{AMU1} holds for a large number of elements of $\pi_1(S)$, namely the ones which are \emph{Euler incompressible}. Viewing the geodesic representative for $1\neq\gamma\in\pi_1(S)$ as an embedded subgraph of $S$, Euler incompressibility means that no Eulerian cycle (i.e. one which visits each edge at most once) bounds a disk in $S$. In particular, they recover one of the main results of~\cite{KoberdaSantharoubane}.

\end{document}